\documentclass[letterpaper, 10 pt, conference]{ieeeconf}
\usepackage{setspace}
\usepackage{amsmath}

\usepackage{amssymb}
\newtheorem{lem}{Lemma}
\newtheorem{thm}{Theorem}
\newtheorem{ass}{Assumption}
\usepackage{color}
\usepackage{graphicx}
\usepackage{subcaption}
\usepackage{wrapfig}
\graphicspath{ {pics/} }
\usepackage{mathptmx}
\allowdisplaybreaks
\IEEEoverridecommandlockouts    
\begin{document}

\title{\LARGE \bf Simultaneous Stabilization
	 of Traffic Flow on Two Connected Roads}  
% Using the Aw-Rascle-Zhang Model		
				
\author{Huan Yu$^{*}$, Jean Auriol $^{\dagger}$, Miroslav Krstic$^{*}$ %<-this % stops a space,
	\thanks{$^{*}$Huan Yu and Miroslav Krstic are with the Department of Mechanical and Aerospace Engineering,
		University of California, San Diego, 9500 Gilman Dr, La Jolla, CA 92093, United States.
		{\tt\small huy015@ucsd.edu, krstic@ucsd.edu}}%
	\thanks{$^{\dagger}$Jean Auriol is with the Department of Petroleum Engineering, University of Calgary, 2500 University Dr NW Calgary, Canada.
		{\tt\small jean.auriol@ucalgary.ca}}%
}
                                           
\maketitle

%		\begin{keyword}                           % Five to ten keywords,  
%			Traffic network; ARZ traffic model; PDE control; Backstepping.               % chosen from the IFAC 
%		\end{keyword}                             % keyword list or with the 
		% help of the Automatica 
		% keyword wizard

		\begin{abstract}                          % Abstract of not more than 200 words.
	In this paper we develop a boundary state feedback control law for a traffic flow network system in its most fundamental form: one incoming and one outgoing road connected by a junction.  The macroscopic traffic dynamics on each road segment are governed by Aw-Rascle-Zhang (ARZ) model, consisting of second-order nonlinear partial differential equations (PDEs) for traffic density and velocity. Different equilibrium road conditions are considered for the connected segments. For stabilization of the stop-and-go traffic congestion on the two roads, we consider a ramp metering located at the connecting junction. The traffic flow rate entering from the on-ramp to the mainline junction is actuated. The objective is to simultaneously stabilize the upstream and downstream traffic to a given spatially-uniform constant steady-state. We design a full state feedback control law for this under-actuated network of two systems of two hetero-directional linear first-order hyperbolic PDEs interconnected through the boundary condition (junction). The exponential stability is validated by numerical simulation. 

		\end{abstract}

\section{Introduction}

Freeway traffic modeling and management has been intensively investigated over the past decades. Motivations behind are to understand the formation of traffic congestion, and further to prevent or suppress instabilities of traffic flow. Macroscopic modeling of traffic dynamics is used to describe the evolution of aggregated traffic state values on road. Traffic dynamics are governed by hyperbolic PDEs. The most widely-used macroscopic traffic PDE models include the first-order Ligthill-Whitham-Richards (LWR) model and the second-order Aw-Rascle-Zhang (ARZ) model~\cite{AW:00}~\cite{Zhang02}. The LWR model corresponds to a conservation law of the traffic density. It can predict the formation and propagation of traffic shockwaves on freeway, but fails to describe the stop-and-go phenomenon~\cite{Daganzo95}~\cite{Klar00}~\cite{Seibold:09}. The oscillations of densities and velocities travel with traffic stream, cause unsafe driving conditions, increased consumptions of fuel and delay of travel time. The second-order ARZ traffic model is developed to describe this common phenomenon. It consists of a set of nonlinear hyperbolic PDEs describing the evolution of the  traffic density and velocity. More recently, the macroscopic modeling of road networks based on the ARZ model has been developed in \cite{Garavello16}~\cite{Herty06} which we consider to use as the system model in this work.

Traffic control strategies are developed and implemented for the traffic management infrastructures, including ramp metering and varying speed limits. Boundary feedback control algorithms are studied for traffic regulation on a freeway segment in  \cite{Bastin16}~\cite{Belletti15}~\cite{Karafyllis19}~\cite{Karafyllis18}~\cite{Yu:18}~\cite{Yu:19}~\cite{LZhang:18}~\cite{LZhang:19}. In authors' previous work \cite{Yu:18}-\cite{Yu:19}, backstepping boundary control laws for ramp metering are designed to suppress the stop-and-go traffic oscillations on freeway either upstream or downstream of the ramp.
In Fig~\ref{ramp}, the traffic flux is actuated through the traffic lights located on the ramp so that the upstream traffic in domain $\mathcal{U}$ or the downstream traffic in domain $\mathcal{D}$ can be controlled. However, the upstream and the downstream traffic can not be stabilized simultaneously and this control design does not consider distinct traffic scenarios for both traffic segments. In this paper, we aim to solve two questions that remained unanswered in the previous work by reformulating the problem in the network setting and providing a more applicable control design approach. The first question is how the ramp metering control of the upstream traffic affects the downstream traffic. The second question is how we can control the downstream and the upstream traffic simultaneously. Furthermore, the control of this basic system of two connected freeway segments will be a important milestone before considering the control problem of the macroscopic modeling of general traffic road network.

\begin{figure}[t]
	\begin{center}
		\includegraphics[height=3.5cm]{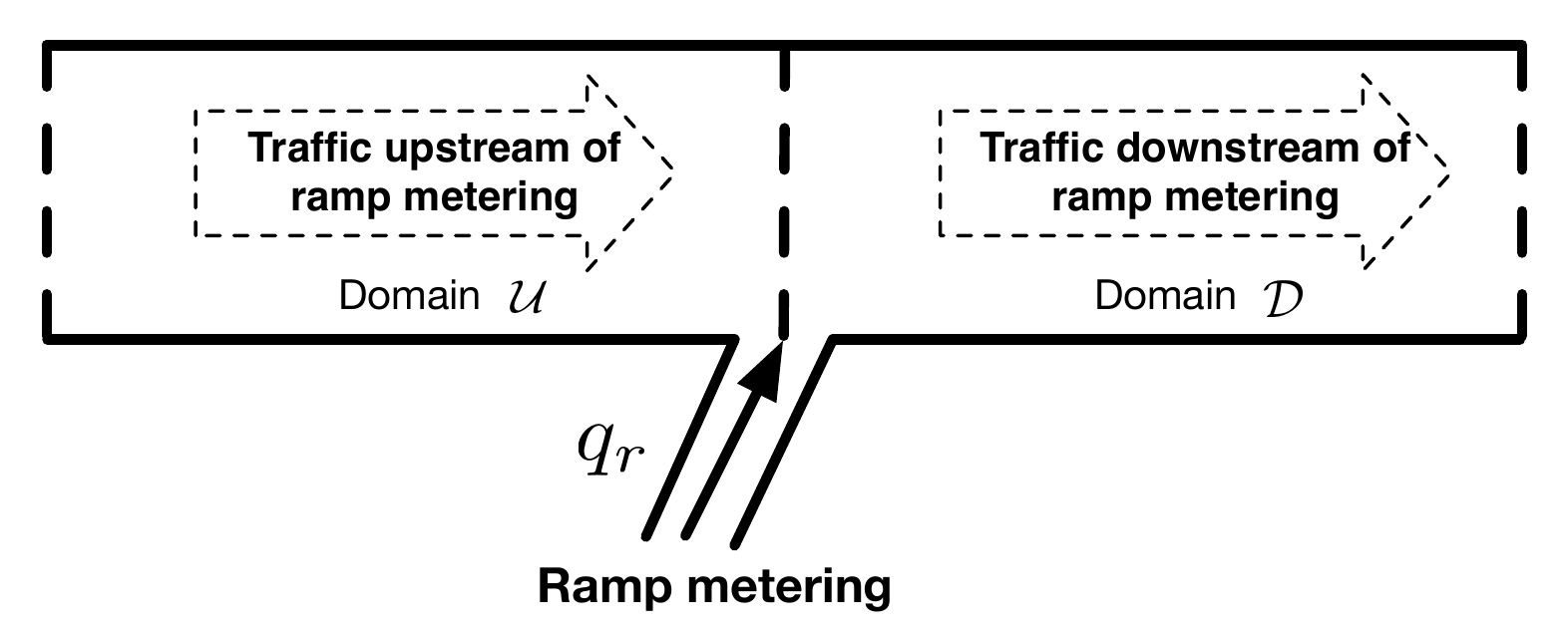}    % The printed column  
		\caption{Boundary feedback control of freeway traffic through ramp metering, the upstream freeway traffic and the downstream traffic of the ramp are simultaneously stabilized.}  % width is 8.4 cm.
		\label{ramp}                                 % Size the figures 
	\end{center}                                  % accordingly.
\end{figure}

In this work, we adopt the traffic network modeling proposed in \cite{Herty06}. The modeling of the junction of two connected roads conserves the mass and the other traffic property as detailed later in the paper. This property is not smooth across the junction in \cite{Garavello16}. In comparison, the solution in \cite{Herty06} is a weak solution (in the sense of the conservative variables of the ARZ model) that guarantees the well-posedness of the closed-loop system for our control design. The considered system of two connected freeway segments can then be rewritten as a network of two interconnected PDE systems coupled through their boundaries. Each subsystem corresponds to a $2\times 2$ coupled hyperbolic systems. Despite the fact that numerous theoretical results in the literature are focused on the boundary control of this class of hyperbolic system based on the backstepping approach~\cite{Anfinsen}~\cite{auriol16}~\cite{coron2013local}~\cite{Deutscher17}~\cite{Vazquez}~\cite{Yu:17}, the control of the network of PDEs remains a challenging research topic. This is due to the fact that in most cases, these systems are underactuated (only the PDE located at one extremity of the network can be actuated). To tackle this problem, multiple approaches have been proposed: PI boundary controllers~\cite{bastin2013exponential,bastin2015stability}, flatness based design of feedforward control laws~\cite{schmuck2011flatness,schmuck2014feed} and more recently backstepping-based control laws~\cite{auriol2019network}.

The main contribution of this paper is to provide an explicit control design that simultaneously stabilizes the traffic flow on two connected road. 
The paper is organized as follows. In Section II, we introduce the system under consideration. In particular, we give the PDEs describing the dynamics of the traffic density and velocity. These equations are then linearized around a given steady-state. A stabilizing state-feedback control law is obtained in Section III for this underactuated system using a backstepping approach. Some simulation results are presented in Section IV. Finally, some concluding remarks are given in Section V.

\section{Problem Statement}

\begin{figure}[t]
\begin{center}
	\includegraphics[height=3.5cm]{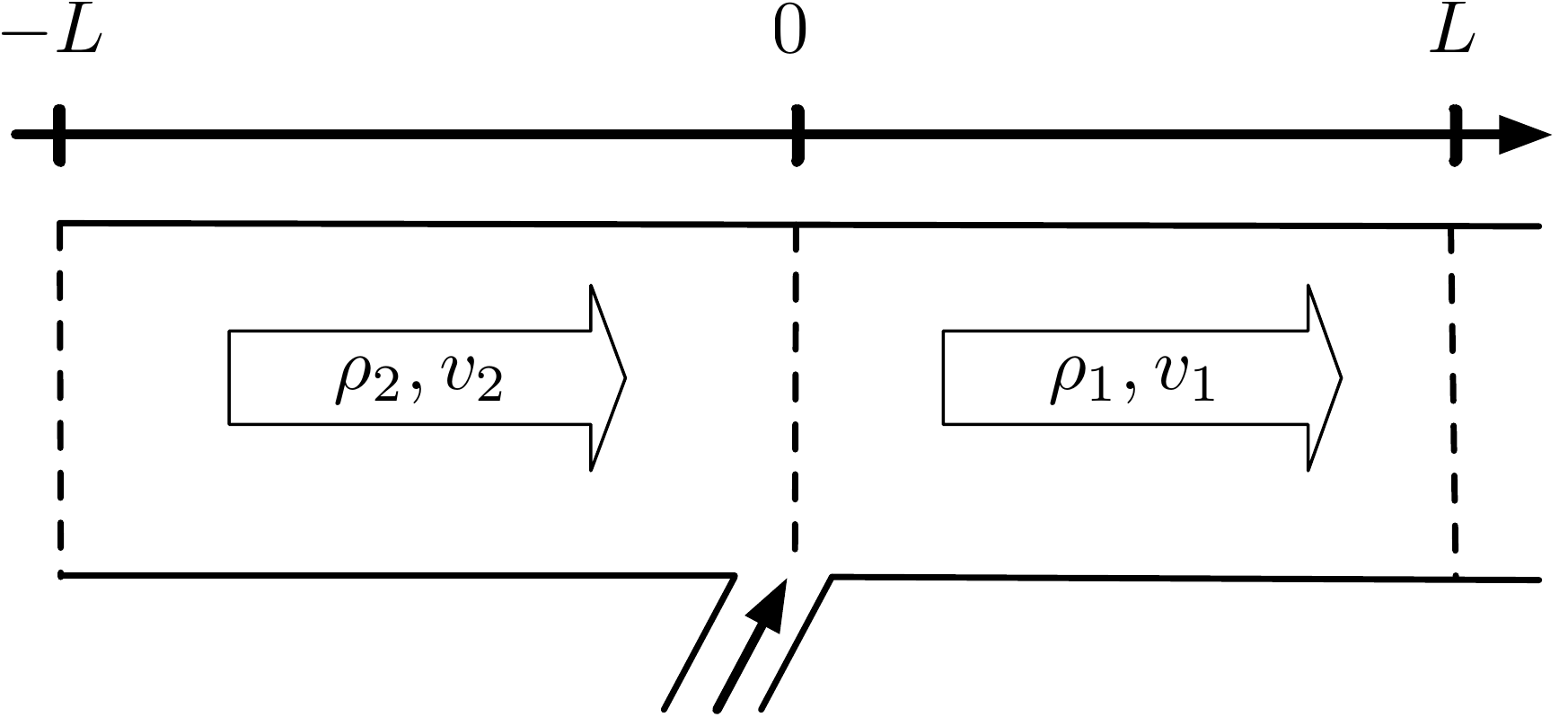}    % The printed column  
	\caption{Traffic flow on an incoming road and an outgoing road connected with a junction, actuation is implemented at the junction.}  % width is 8.4 cm.
	\label{midinput}                                 % Size the figures 
\end{center}                                  % accordingly.
\end{figure}

We consider a road network that contains two connected road segments with unidirectional traffic flow. The road conditions and properties are different for the two segment. The two segments are assumed to have the same length $L$ for simplicity. For segments with different length, we rescale the equations describing the traffic system and the control design we propose can be directly applied. The outgoing road segment is defined on $[0,L]$ while the incoming road segment is defined on $[-L,0]$. These two segments are connected at the junction through the boundary $x=0$. The traffic dynamics of road segments are described with PDE models and the junction between the two segments implies boundary conditions for the PDE model. This allows the existence of weak solutions for the traffic network problem~\cite{Herty06}.

\subsection{ARZ PDE model}
The evolution of traffic density $\rho_1(x,t)$ and velocity $v_1(x,t)$ (with $(x,t) \in [0,L]\times [0, \infty)$) on the outgoing road segment and traffic density  $\rho_2(x,t)$ and velocity $v_2(x,t)$  ($(x,t) \in [-L,0]\times [0, \infty)$) on the incoming road segment are modeled by the following ARZ model.

\begin{align}
	\partial_t \rho_i + \partial_x (\rho_i v_i) =& 0, \label{PDE_ARZ_1}\\
	\partial_t (\rho_i ( v_i + p_i(\rho_i))) + \partial_x (\rho_i v_i ( v_i + p_i(\rho_i))) =& -\frac{\rho_i(v_i -V(\rho_i)))}{\tau_i}, \label{PDE_ARZ_2}
\end{align} 
where $i \in \{1,2\}$ represents either the outgoing road segment or the incoming road segment. The traffic pressure $p_i(\rho_i)$ is defined as an increasing function of the density 
\begin{align}
p_i(\rho_i) = c_i \rho^{\gamma_i}_i,\label{prho}
\end{align} 
where $\gamma_i,c_i \in \mathbb{R^+}$ is defined as
\begin{align}
c_i = \frac{v_m}{\rho_{m,i}^{\gamma_i}}.\label{pc}
\end{align}
The coefficient $\gamma_i$ represents the overall drivers' aggressiveness, the positive constant $v_m$ represents the maximum velocity and the positive constant $\rho_{m,i}$ is the maximum density defined as the number of vehicles per unit length.
The equilibrium density-velocity relation $V_i(\rho_i)$ is given in the form of Greenshield's model 
\begin{align}
V_i(\rho_i)=v_m\left(1-\left(\frac{\rho_i}{\rho_{m,i}}\right)^{\gamma_i}\right). \label{vf}
\end{align}
Given the definitions of \eqref{prho},\eqref{pc} and \eqref{vf}, the following relation between $V_i(\rho_i)$ and $p_i(\rho_i)$ is satisfied on both segments.
\begin{align}
V(\rho_i) + p_i(\rho_i) = v_m, \label{vm}
\end{align}
where the marginal stability is satisfied for each segment. The linear stability analysis of the system can be found in \cite{Yu:19}. We define the following variable
\begin{align}
	w_i= v_i + p_i(\rho_i), \label{w}
\end{align} which is interpreted as traffic "friction" or drivers' property which transports in traffic flow with the vehicle velocity. This property represents the heterogeneity of traffic flow with respect to the equilibrium density-velocity relation $V_i(\rho_i)$. The maximum velocity $v_m$ is assumed to be the same for the two road segments while the maximum density $\rho_{m,i}$ and coefficient $\gamma_i$ are allowed to vary in the different segments. The positive constant $\tau_i$ is the relaxation time that represents the time scale for traffic velocity $v_i$ adapting to the equilibrium density-velocity relation $V_i(\rho_i)$. 

Finally, we denote the traffic flow on each road as
\begin{align}
	 q_i=\rho_iv_i.
\end{align}
The equilibrium flow and density relation, also known as the fundamental diagram, is then given by
\begin{align}
	Q_i(\rho_i) =& \rho_iV(\rho_i)
=\rho_iv_m\left(1-\left(\frac{\rho_i}{\rho_{m,i}}\right)^{\gamma_i}\right). \label{Qrho}
\end{align}

We assume that the equilibrium traffic relation is different for the two segments due to the change of road situations. The illustration is given in Fig \ref{fig:equ}. The critical density $\rho_c$ segregates the free and congested regimes of traffic states. For the fundamental diagram in \eqref{Qrho}, the critical density is given by
$\rho_{c,i}=\frac{\rho_{m,i}}{(1+\gamma_{i})^{1/\gamma_{i}}}.$
The traffic is in the free regime when the density satisfies $\rho < \rho_c$. The traffic is defined as the congested traffic when the density satisfies $\rho > \rho_c$. The traffic flux reaches its maximum value at the critical density $Q(\rho_c)$ which is also referred as the road capacity. 

In this work, we are concerned with the congested traffic and assume that the equilibrium of both segments are in the congested regime. We consider the situation that the upstream road segment for $x \in [-L,0]$ has more lanes than the downstream road segment for $x \in [0,L]$. Therefore, the maximum density $\rho_{m,2} >\rho_{m,1} $ and the maximum speed limit $v_m$ is assumed to be the same for the two segments. The upstream road capacity that is the maximum value of $Q_2(\rho_c)$ is reduced in the downstream for $Q_1(\rho_c)$, due to the change of road conditions from the segment 2 to the segment 1.  
\begin{figure}[t!]	
	\begin{subfigure}{0.23\textwidth}
		\includegraphics[width=0.9\linewidth]{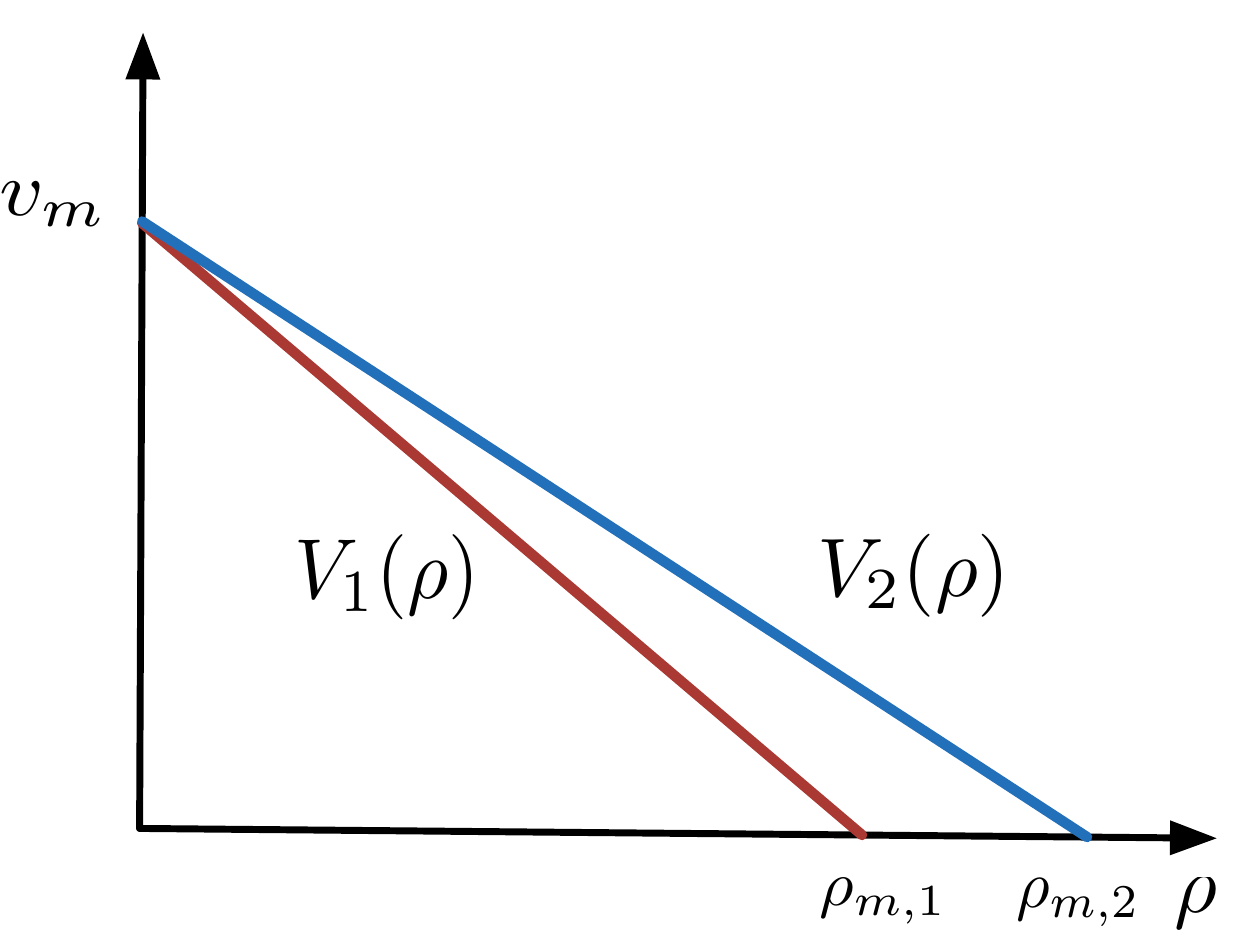} 
	\end{subfigure}
	\begin{subfigure}{0.23\textwidth}		\includegraphics[width=0.9\linewidth]{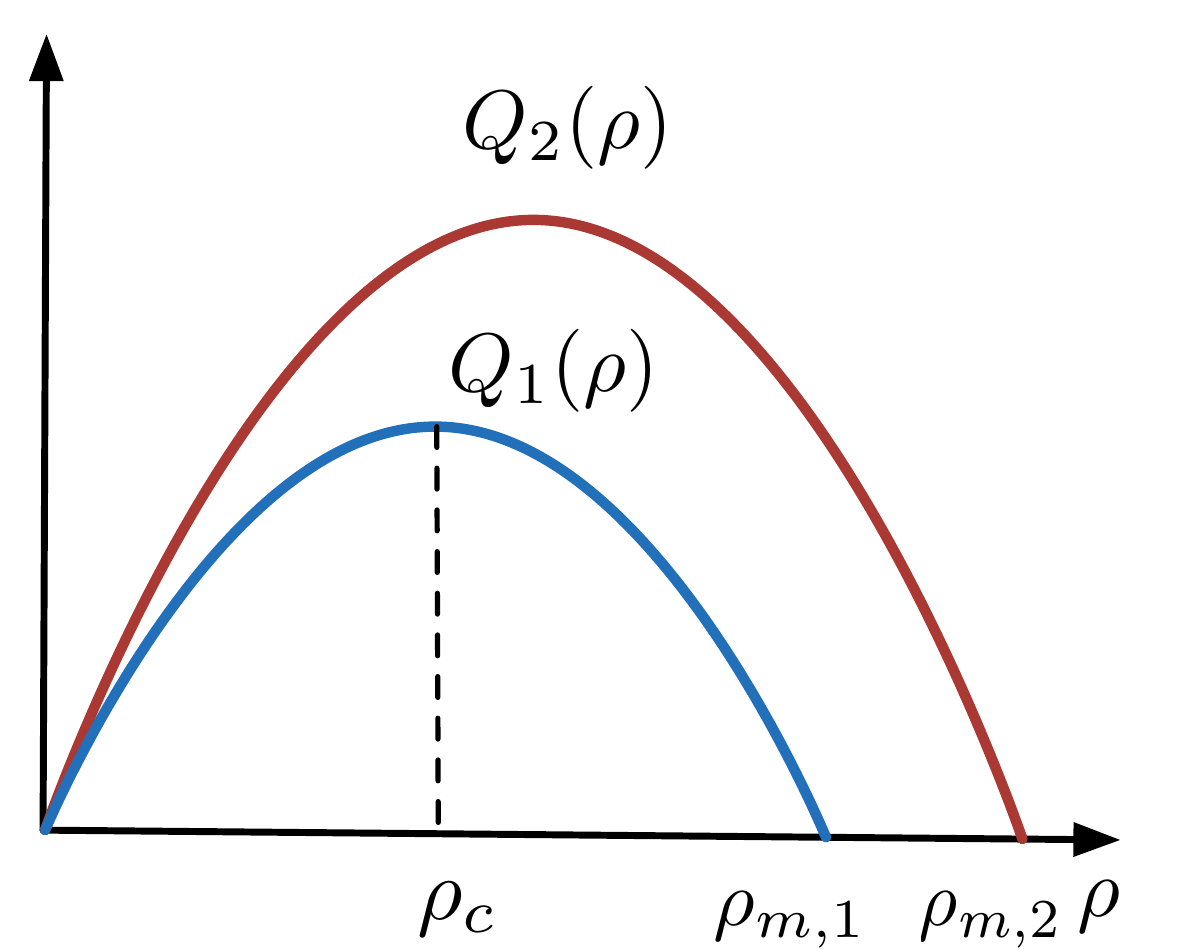}
	\end{subfigure}
	\caption{The equilibrium density and velocity relation $V_i(\rho)$ on the left, the equilibrium density and flux relation $Q_i(\rho)$ on the right}
	\label{fig:equ}
\end{figure}
\subsection{Boundary conditions}

In this paper we consider the weak solution of the network \eqref{PDE_ARZ_1}-\eqref{PDE_ARZ_2}. Regarding the boundary conditions connecting the two PDE systems, the Rankine-Hugoniot condition is satisfied at the junction. This condition implies piecewise smooth solutions and corresponds to the conservation of the mass and of the drivers' properties defined in \eqref{w} at the junction, i.e. 
\begin{align}
	\rho_1 v_1(0^-,t) = & \rho_2 v_2(0^+,t),\\
	\rho_1 v_1 w_1(0^-,t) = & \rho_2 v_2 w_2(0^+,t),	
\end{align}
Thus, we assume that the flux and drivers' property are continuous across the boundary conditions at $x=0$, that is
\begin{align}
	\rho_1(0,t) v_1(0,t) =& \rho_2(0,t) v_2(0,t), \label{bc1}\\
	 w_2(0,t) = &  w_1(0,t).\label{bc2}
\end{align}
For open-loop system, we assume a constant incoming flow $q^\star$ entering the inlet boundary $x= -L$ and a constant outgoing flow $q^\star$ at the outlet boundary for $x= L$:
\begin{align}
	q_2(-L,t) =& q^\star, \label{q_inlet}\\
	q_1(L,t)=& q^\star,\label{q_outlet}
\end{align} 
The control problem we solve consists in stabilizing the traffic flow in both the incoming and outgoing road segments around given steady-states. We consider the actuator $U_0(t)$ with the ramp metering located at the junction boundary $x=0$, controlling the traffic flow entering from the junction to the mainline road. Given the flux continuity condition, we have the following boundary condition at the junction
\begin{align}
	\rho_1(0,t) v_1(0,t) =& \rho_2(0,t) v_2(0,t) + U_0(t), \label{mid_input}
\end{align}
where the downstream segment flow consists of the incoming flow from the mainline upstream segment and the actuated traffic flow from the on-ramp.

\subsection{Steady states $(\rho_1^\star,v_1^\star,\rho_2^\star,v_2^\star)$}
The control objective is to stabilize the traffic flow in the two segments around the steady states. These arbitrary steady states $(\rho_1^\star,v_1^\star)$, $(\rho_2^\star,v_2^\star)$ are chosen such that the boundary conditions \eqref{bc1} and \eqref{bc2} are satisfied, i.e.

\begin{align}
	\rho_1^\star v_1^\star =& \rho_2^\star v_2^\star = q^\star, \label{con_q}\\
	w_1^\star =& w_2^\star = v_m,\label{con_w}
\end{align}
where the steady state velocities satisfy the equilibrium density-velocity relation $v_i^\star =V_i(\rho_i)$. The constant flux in \eqref{con_q}
\begin{align}
Q_1(\rho_1^\star) = Q_2(\rho_2^\star),
\end{align}
and the definition of $Q_i(\rho_i)$ in \eqref{Qrho} yields the following relation for the steady state densities of the two segments
\begin{align}
 \frac{\rho_1^\star - (\rho^\star_1)^{\gamma_1 +1} }{\rho_2^\star - (\rho^\star_2)^{\gamma_2 +1}} = \frac{\rho_{m,1}^{\gamma_1}}{\rho_{m,2}^{\gamma_2}}. \label{rho_relation}
\end{align}
According to \eqref{w}, the constant driver's property in \eqref{con_w} implies that we have the same maximum velocity $v_m$ for the two segments (which corresponds to our initial assumption):
\begin{align}
	v_1^\star + p_1^\star = v_2^\star + p_2^\star = v_m.
\end{align}
Note that when vehicles' property follows the equilibrium relation $v_i=V(\rho_i)$, the above relation always holds given \eqref{vm}. In summary, we first choose the steady states density $\rho_1^\star$ and $\rho_2^\star$ such that the relation in \eqref{rho_relation} is satisfied. Then the steady states velocities are obtained as $v_i^\star=V(\rho_i^\star)$.

\subsection{Linearized model in the Riemann coordinates $(\tilde w_1,\tilde v_1,\tilde w_2, \tilde v_2)$}
%The linearized state variables are chosen as the Riemann variables $(\tilde w_i, \tilde v_i)$ and are defined as 
%\begin{align}
%	\tilde w_i(x,t) =& {w}_i(x,t) - w_i^\star,\\
%	\tilde v_i(x,t) =& {v}_i(x,t) - v_i^\star. 
%\end{align}

We linearize the ARZ based traffic network model $(\rho_i, v_i)$ in \eqref{PDE_ARZ_1}, \eqref{PDE_ARZ_2} with the boundary conditions \eqref{bc1}, \eqref{bc2}, \eqref{q_inlet}, \eqref{q_outlet} around the steady states $(\rho_i^\star,v_i^\star)$ defined in the previous section. In order to obtain a simpled model for control design, the linearized model is given in the following Riemann variables defined as  
\begin{align}
\tilde w_i =&\frac{\gamma_i p_i^\star}{q^\star} (\rho_i v_i - \rho_i^\star v_i^\star) + \frac{1}{r_i}( v_i - v_i^\star),\\
\tilde v_i =& v_i - v_i^\star, \label{vtv}
\end{align}
where the constant coefficients $r_i$ are defined as
\begin{align}
r_i = - \frac{v_i^\star}{\gamma_i p_i^\star-v_i^\star}. \label{gamma}
\end{align}
Given the controlled boundary at $x=0$ in \eqref{mid_input} and steady states condition in \eqref{con_q} ,
\begin{align}
	\tilde q_1(0,t) = \tilde  q_2(0,t) + U_0(t). 
\end{align}

Then we obtain the linearized model with boundary conditions 
\begin{align}
\partial_t \tilde w_i +v_i^\star \partial_x\tilde w_i&=-\frac{1}{\tau_i} \tilde w_i ,\label{eq_syst_eq_1}\\
\partial_t \tilde v_i - (\gamma_i  p_i ^\star - v_i ^\star )   \partial_x \tilde v_i &=-\frac{1}{\tau_i } \tilde w_i ,\\ 
%\tilde w_1 (0,t)=&-r_1 \bar v_1 (0,t),\\
\tilde v_1 (L,t) &= r_1\tilde w_1(L,t),\\
   \tilde w_1(0,t) &= \tilde w_2(0,t) ,\\	
    \tilde w_2(-L,t) &= \frac{1}{r_2}\tilde v_2(-L,t),\\
  \notag    \tilde v_2(0,t) &= \frac{r_2}{r_1}\tilde v_1(0,t) + \frac{r_1-r_2}{1-r_1} \tilde w_2(0,t)  \\
      & +\frac{v_2^\star({1-r_2} )}{q^\star} U_0(t), \label{eq_syst_bound_1}
%    \bar v_2(0,t) =& \frac{1}{r_2} (w_2(0,t) - w_1(0,t)) + \frac{r_1}{r_2}\bar v_1(0,t),
\end{align}
Detailed calculations regarding the linearization can be obtained following \cite{Yu:19}.

For the congested regime of traffic flow, $\rho_i^\star > \frac{\rho_{m,i}}{(1+\gamma_i)^{1/\gamma_i}}$ is satisfied so that the characteristic speed $\gamma_i  p_i ^\star - v_i ^\star > 0$. The velocity variations $\tilde v_i(x,t)$ transport upstream which means the action of velocity acceleration or deceleration is repeated from the leading vehicle to the following vehicle.  
The following inequality is satisfied for the characteristic speeds ratio defined in \eqref{gamma},
\begin{align}
0<r_i<1.
\end{align}
The more congested of the traffic, the smaller of the ratio constant $r_i$. The control diagram for the closed-loop system \eqref{eq_syst_eq_1}-\eqref{eq_syst_bound_1} is given in Fig. \ref{diagram}.
\begin{figure}[t]
\begin{center}
	\includegraphics[height=4cm]{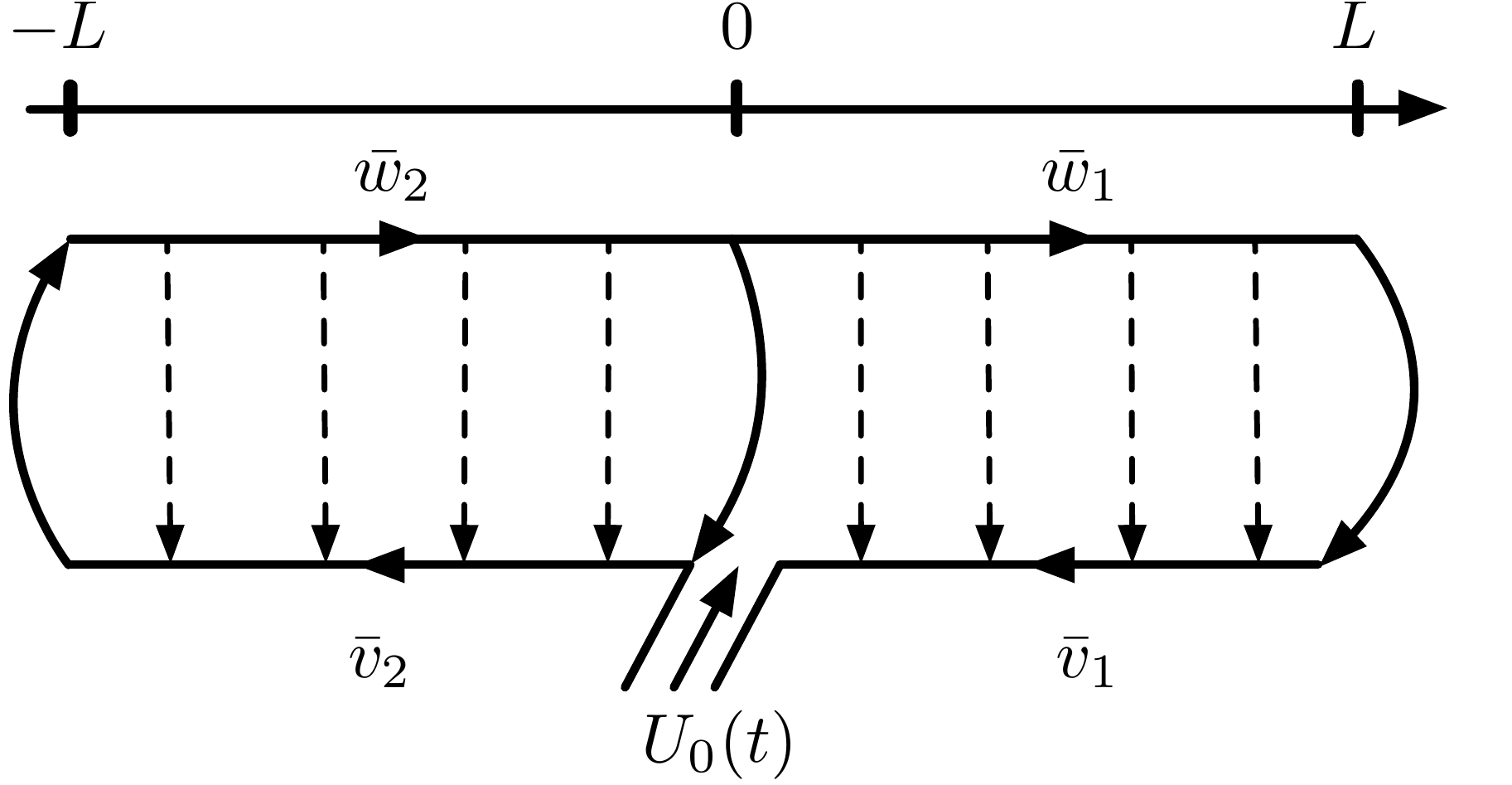}    % The printed column  
	\caption{Control diagram for the closed-loop system}  % width is 8.4 cm.
	\label{diagram}                                 % Size the figures 
\end{center}                                  % accordingly.
\end{figure}

Using a spatial transformation, we get rid of the diagonal terms $-\frac{1}{\tau_i} \tilde w_i $ that appear in the two equations describing the evolution of the state $w_i$. More precisely, we define for all $x \in [-L,0]$ and all $t >0$ the state $\bar{w}_2$ as
\begin{align}
	\bar w_2(x,t)= \exp\left(\frac{x}{\tau_2 v_2^\star}\right) \tilde w_2(x,t).
\end{align}
Similarly, the state $\bar{w}_1$ is defined for all $x \in [0,L]$ and all $t >0$ by
 \begin{align}
	\bar w_1(x,t)= \exp\left(\frac{x}{\tau_1 v_1^\star}\right) \tilde w_1(x,t).
\end{align}
One can easily check that with such a change of variable, the system \eqref{eq_syst_eq_1}-\eqref{eq_syst_bound_1} rewrites (for $i \in \{1,2\}$) as
\begin{align}
	\partial_t \bar w_i +v_i^\star \partial_x \bar w_i&=0 , \label{orig_1}\\
\partial_t \tilde  v_i - (\gamma_i  p_i ^\star - v_i ^\star )   \partial_x \tilde v_i &=c_i(x)\bar  w_i ,  \label{orig_2}\\ 
%\tilde w_1 (0,t)=&-r_1 \bar v_1 (0,t),\\
\tilde v_1 (L,t)&= r_1  \exp\left(-\frac{L}{\tau_1 v_1^\star}\right) \bar w_1(L,t),\\
\bar w_1(0,t) &= \bar w_2(0,t) ,\\	
\bar w_2(-L,t) &= \exp\left(\frac{-L}{\tau_2 v_2^\star}\right)  \frac{1}{r_2}\bar v_2(-L,t),\\
\notag \tilde v_2(0,t) &= \frac{r_2}{r_1}\tilde v_1(0,t) + \frac{r_1-r_2}{1-r_1} \bar w_2(0,t) \\&+ \frac{v_2^\star({1-r_2} )}{q^\star} U_0(t), \label{orig_end}
\end{align}
where the spatially varying coefficients $c_i(x)$ are defined as
$
	c_i(x) = -\frac{1}{\tau_i}\exp\left(-\frac{x}{\tau_i v_i^\star}\right).
$
The corresponding initial condition are denoted~$(\tilde{v}_0)_i=\tilde{v}_i(\cdot,0)$ and~$(\bar{w}_0)_i=\bar{v}_i(\cdot,0)$. The objective is to design the control law~$U_0$ to stabilize the system~\eqref{orig_1}-\eqref{orig_end} in the sense of the $L^2$-norm. Such an interconnected system has already been considered in~\cite{auriol2019network} in the case of an actuator located at one of the extremity of the network. It has been proved in \cite{logemann1996conditions} that a system can be delay-robustly stabilized only if its open-loop transfer function has a finite number of zeros on the complex right half plane. For the considered class of linear hyperbolic system, it has been proved in \cite{auriol2019mapping} that such a condition is equivalent to requiring \eqref{orig_1}-\eqref{orig_end} with zero in-domain couplings (i.e. $c_1 \equiv c_2 \equiv 0$) to be exponentially stable in open-loop. Necessary and sufficient stability conditions to guarantee such an open-loop stability can be obtained by writing the corresponding characteristic equations. However for the case of system \eqref{orig_1}-\eqref{orig_end}, a simpler condition has been given in~\cite{coron2008dissipative} in the form of the following Assumption.
\begin{ass} \label{Assumption_main}
The boundary couplings of the system \eqref{orig_1}-\eqref{orig_end} are such that
\begin{align}
Sp_1\begin{pmatrix}
	0&0&0&1 \\
	0&0& \frac{r_2}{r_1}&\frac{r_1-r_2}{1-r_1}\\
	 \exp\left(\frac{-L}{\tau_1 v_1^\star}\right)  \frac{1}{r_2} &0&0&0\\
	0& r_1\exp\left(\frac{-L}{\tau_2 v_2^\star}\right) &0&0
\end{pmatrix}<1, \label{radius}
\end{align}
where for a matrix $H\in \mathcal{M}_{4,4}(\mathbb{R})$ (with~$\mathcal{M}_{r,s}(\mathbb{R})$ denoting the set of real matrices with~$r$ rows and~$s$ columns) we have
\begin{align*}
\text{Sp}_1(H)&= \text{Inf}\{ || \Delta H \Delta^{-1}||;~\Delta \in \mathcal{D}_{4,+} \},
\end{align*}
with~$\mathcal{D}_{4,+}~$ denoting the set of~$4\times 4$ real diagonal matrices with strictly positive diagonal elements. 
\end{ass}
The condition of spectral radius $Sp_1$ in \eqref{radius} leads us to the following inequality 
\begin{align}
	\sqrt{\frac{a + \sqrt{a^2 + 4b }}{2}} <1, \label {condition}
\end{align}
where $a, b \in \mathbb{R}^+$ are defined as
\begin{align}
a = & \frac{r_1-r_2}{1-r_1}r_1\exp\left(\frac{-L}{\tau_2^\star v_2^{\star}}\right),\\
b = & \exp\left(-\frac{L}{\tau_2^\star v_2^{\star}}-\frac{L}{\tau_1^\star v_1^{\star}}\right).
\end{align}
Here we assume that $r_1 > r_2$ is satisfied which corresponds to the fact that downstream traffic in segment 1 is more congested than the upstream traffic in segment 2. Note that this condition is usually satisfied for the class of traffic networks considered in this paper. The condition in \eqref{condition} becomes the following and it is guaranteed 
\begin{align}
	\exp\left(-\frac{L}{\tau^\star v^{\star}}\right) < 1.
\end{align}
if we consider the traffic conditions in the two segments to be the same $r_1 = r_2$.

%%%%%%%%%%%%%%%%%%%%%%%%
\section{State feedback Control Design}
In this section we design a full-state feedback law that guarantees the stabilization of the system \eqref{eq_syst_eq_1}-\eqref{eq_syst_bound_1}. Our approach is based on the backstepping methodology. Using two Volterra transformations we map the original underactuated system to a target system for which the in-domain coupling terms $c_1$ and $c_2$ have been moved at the actuated boundary in the form of integral couplings. We can then use the actuation $U_0(t)$ to eliminate these terms, the resulting system being exponentially stable due to Assumption~\ref{Assumption_main}. As such a control law does not modify the boundary couplings, robustness margins are preserved (see~\cite{auriol2018delay,auriol2019mapping} for details). 
%%%
\subsection{Feedback law with flow control input from~$x=0$}
We consider the following backstepping transformations
\begin{align}
\alpha_i(x,t)=&\bar{w}_i(x,t), \\
\notag\beta_1(x,t)=&\tilde{v}_1(x,t)-\int_{x}^L K_1^{vw}(x,\xi)\bar{w}_1(\xi,t) d\xi\\&-\int_{x}^L K_1^{vv}(x,\xi)\tilde{v}_1(\xi,t) d\xi,\label{Back_con0_1}\\
\notag\beta_2(x,t)=&\tilde{v}_2(x,t)-\int_{-L}^x K_2^{vw}(x,\xi)\bar{w}_2(\xi,t) d\xi\\&-\int_{-L}^x K_2^{vv}(x,\xi)\tilde{v}_2(\xi,t) d\xi, \label{Back_con0_2}
\end{align}
where the kernels~$K_1^{vw}$ and~$K_1^{vv}$ are $L^\infty$ functions defined on the set~$\mathcal{T}_1=\{(x,\xi)~\in~[0,L]^2,~\xi \geq x\}$, while the kernels~$K_2^{vw}$ and~$K_2^{vv}$ are $L^\infty$ functions defined on the set~$\mathcal{T}_2=\{(x,\xi)~\in~[-L,0]^2,~\xi \leq x\}$. On their corresponding domains of definition, they satisfy the following set of PDEs:
\begin{align}
(\gamma_ip_i^\star -v_i^\star)\partial_x K_i^{vw}-v_i^\star\partial_\xi K_i^{vw}&=c_i(\xi)K_i^{vv}, \label{eq_kernel_1} \\
\partial_x K_i^{vv}(x,\xi) +\partial_\xi K_i^{vv}(x,\xi)&=0, 
\end{align}
along with the following boundary conditions
\begin{align}
&K_1^{vw}(x,x)=\frac{c_1(x)}{\gamma_1p_1^\star}, \quad K_2^{vw}(x,x)=-\frac{c_2(x)}{\gamma_2p_2^\star},\\
&K_1^{vv}(x,L)=-\exp\left(\frac{L}{\tau_1 v_1^\star}\right) K_1^{vw}(x,L), \\
&K_2^{vv}(x,-L)=-\exp\left(\frac{-L}{\tau_2 v_2^\star}\right)  K_2^{vw}(x,-L).\label{eq_kernel_final}
\end{align}
The well-posedness of this kernel PDE-system is guaranteed by the following lemma.
\begin{lem}\cite{coron2013local} Consider system~\eqref{eq_kernel_1}-\eqref{eq_kernel_final}. There exists a unique solution~$K_1^{vw}$,~$K_1^{vv}$ in~$L^\infty(\mathcal{T}_1)$ and~$K_2^{vw}$,~$K_2^{vv}$ in~$L^\infty(\mathcal{T}_2)$.
\end{lem}
The transformation~\eqref{Back_con0_1}-\eqref{Back_con0_2} maps the original system~\eqref{orig_1}-\eqref{orig_end} to the decoupled target system
\begin{align}
	\partial_t \alpha_i +v_i^\star \partial_x \alpha_i=&0 , \label{targ_1}\\
\partial_t \beta_i - (\gamma_i  p_i ^\star - v_i ^\star )   \partial_x \beta_i =&0 , \label{targ_beta}\\ 
%\tilde w_1 (0,t)=&-r_1 \bar v_1 (0,t),\\
\beta_1 (L,t)=& r_1  \exp\left(-\frac{L}{\tau_1 v_1^\star}\right)  \alpha_1(L,t),\\
\alpha_1(0,t) =& \alpha_2(0,t) ,\\	
\alpha_2(-L,t) =& \exp\left(\frac{-L}{\tau_2 v_2^\star}\right)  \frac{1}{r_2}\beta_2(-L,t),\label{targ_a_b_2} \\
\beta_2(0,t) =& \frac{r_2}{r_1}\beta_1(0,t) + \frac{r_1-r_2}{1-r_1} \alpha_2(0,t). \label{targ_end}
\end{align}
The controlled boundary \eqref{targ_end} is obtained by defining the control input~$U_0(t)$ as
\begin{align}
&U_0(t)=\frac{q^\star}{v_2^\star (1-r_2)}\Bigg(\!\!\!\int_{-L}^0 \!\!\! K_2^{vw}(0,\xi)\bar{w}_2(\xi,t)+K_2^{vv}(0,\xi)\tilde{v}_2(\xi,t) d\xi \nonumber \\
&-\frac{r_2}{r_1}\int_0^L K_1^{vw}(0,\xi)\bar{w}_1(\xi,t)+K_1^{vv}(0,\xi)\tilde{v}_1(\xi,t) d\xi\Bigg). \label{state-feedback}
\end{align}
We have the following theorem.
\begin{thm}
Consider the PDE system \eqref{orig_1}-\eqref{orig_end} with the feedback law $U_0$ defined in \eqref{state-feedback}. Then for any $L^2$ initial condition $(\bar{w}_i(0,\cdot), \tilde{v}_i(0,\cdot))$ the system \eqref{orig_1}-\eqref{orig_end} exponentially converges to 0. 
\end{thm}
\begin{proof}
 Using the method of characteristics (see \cite{auriol2019mapping} for details), it is possible to express the state~$\beta_2(L,t)$ as the solution of the difference equation
\begin{align*}
\beta_2(0,t)&= \exp\left(\frac{-L}{\tau_2 v_2^\star}\right) \frac{r_1-r_2}{(1-r_1)r_2}\beta_2(0,t-\kappa_2)\\
&+\exp\left(\frac{-L}{\tau_2 v_2^\star}\right)\exp\left(-\frac{L}{\tau_1 v_1^\star}\right)  \beta_2(0,t-\kappa_1-\kappa_2),
\end{align*}
where $\kappa_i=\frac{1}{v_i^\star}+\frac{1}{\gamma_ip_i^\star-v_i^\star}$. This difference system is exponentially stable due to Assumption~\ref{Assumption_main} (see \cite{auriol2019mapping}). Then, it implies that~$\beta_2(0,t)$ converges to zero. Using the transport structure of~\eqref{targ_1}-\eqref{targ_end} , we have the convergence of ($\alpha_i,\beta_i)$ to zero.
Due to the invertibility of the Volterra transformations~\eqref{Back_con0_1}-\eqref{Back_con0_2}, the systems \eqref{orig_1}-\eqref{orig_end} and \eqref{targ_1}-\eqref{targ_end} have equivalent stability properties. This implies the exponential stability of \eqref{orig_1}-\eqref{orig_end}.
\end{proof}

%%%%%%%%%%%%%%%%%%
\section{Simulation results}
The length of each freeway segment is chosen to be $L= 2\; \rm{km}$ so the total length of the two connected segments are $4\; \rm{km}$. The maximum speed limit is $v_m= 45\; \rm m/s$. We consider six lanes for the upstream freeway segment 2. Assuming the average vehicle length is $5\; \rm{m}$ plus the minimum safety distance of $50\%$ vehicle length, the maximum density of the road is obtained as $\rho_{m,2}= 6/ 7.5 \;\rm vehicles /m = 800 \;\rm vehicles /km $. The downstream segment has five lanes thus its maximum density is $\rho_{m,1}= 666 \;\rm vehicles /km$. We take $\gamma_i=1$.  The steady states $(\rho_1^\star, v_1^\star)$ and $(\rho_2^\star, v_2^\star)$ are chosen respectively as $(456\;\rm vehicles /km,14\;m/s)$ and $(500 \;\rm vehicles /km,17 \;m/s)$, both of which are in the congested regime and satisfy \eqref{con_q} and \eqref{con_w}. The equilibrium steady state of the downstream road is more congested than the upstream road with higher density and lower velocity. The relaxation time of each segments are $\tau_1= 120 \; \rm s$ and $\tau_2=90 \; \rm s$. We use sinusoid initial conditions. The closed-loop simulation show that the exponential convergence to the steady states are achieved simultaneously for the upstream and downstream segments in in Fig ~\ref{close_rho} and Fig ~\ref{close_v}. The temporal evolutions of the traffic density and velocity states at junction are highlighted with red lines. 
\begin{figure}[t!]
	\includegraphics[width=9.6cm]{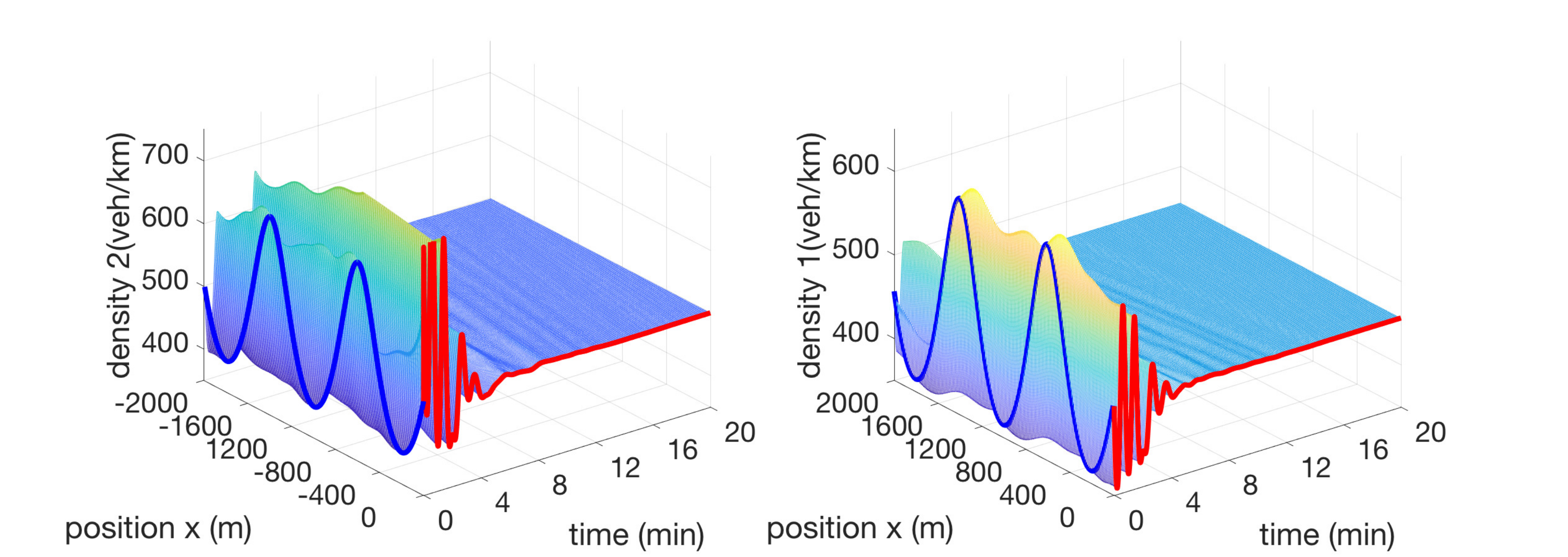}
	\centering
	\caption{Density evolution of upstream and downstream of the ramp metering}
	\label{close_rho}
\end{figure}

\begin{figure}[t!]
	\includegraphics[width=9.4cm]{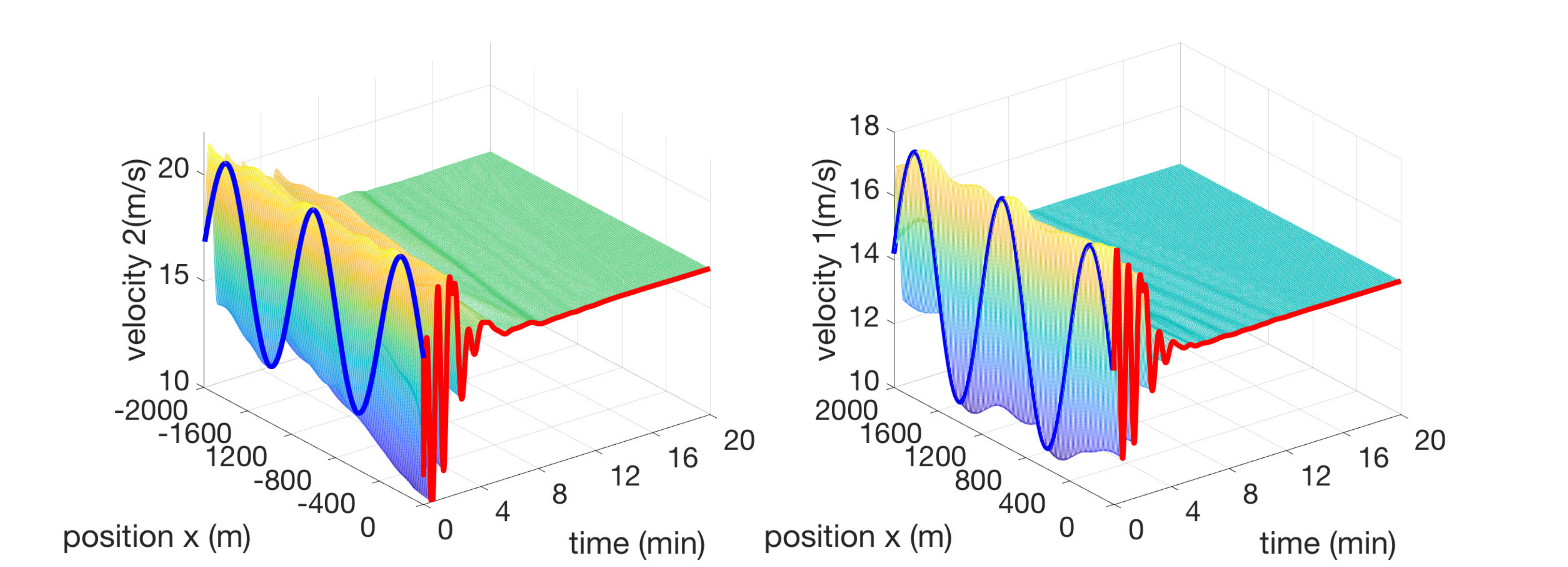}
	\centering
	\caption{Velocity evolution of upstream and downstream of the ramp metering}
	\label{close_v}
\end{figure} 
%Note that in simulation results, maybe not only include the stabilization result. Other performance indicators should be compared for open-loop and closed-loop, such as fuel consumption, total travel time. 

%%%%%%%%%%%%%%%%%%%%
\section{Concluding remarks}
We design a stabilizing control law that guarantees the simultaneous stabilization of the traffic flow on two connected roads around given steady states. The flow actuation is realized with the ramp metering at the junction. Our approach is based on the backstepping methodology. This is a first step towards the stabilization of road networks. We will consider in future work the design of an observer (in view of output-feedback stabilization) for this class of system.

\bibliographystyle{plain}        % Include this if you use bibtex 

\end{document}